\documentclass[12pt,reqno]{amsart}

\usepackage{amssymb}
\usepackage{verbatim}
\usepackage{graphicx}
\usepackage[active]{srcltx}

\newcommand{\cC}{{\mathcal C}}
\newcommand{\cF}{{\mathcal F}}

\newcommand{\bet}{\beta}
\newcommand{\eps}{\varepsilon}
\newcommand{\lam}{\lambda}
\renewcommand{\phi}{\varphi}

\newcommand{\Del}{\Delta}

\newcommand{\longc}{,\dotsc,}
\newcommand{\longp}{+\dotsb+}

\newcommand{\seq}{\subseteq}
\newcommand{\stm}{\setminus}
\newcommand{\est}{\varnothing}

\renewcommand{\)}{\right)}
\renewcommand{\(}{\left(}

\newcommand{\prt}{\partial}

\newtheorem{claim}{Claim}
\newtheorem{lemma}{Lemma}
\newtheorem{theorem}{Theorem}
\newtheorem{corollary}{Corollary}

\newcommand{\refl}[1]{\ref{l:#1}}
\newcommand{\refm}[1]{\ref{m:#1}}
\newcommand{\reft}[1]{\ref{t:#1}}
\newcommand{\refc}[1]{\ref{c:#1}}
\newcommand{\refs}[1]{\ref{s:#1}}
\newcommand{\refe}[1]{\eqref{e:#1}}
\newcommand{\refb}[1]{\cite{b:#1}}

\title[Approximate convexity and an edge-isoperimetric estimate]%
  {Approximate convexity\\ and an edge-isoperimetric estimate}

\author{Vsevolod F. Lev}
\email{seva@math.haifa.ac.il}
\address{Department of Mathematics, The University of Haifa at Oranim,
  Tivon 36006, Israel}

\subjclass[2010]%
  {Primary: 39B62;    
   secondary: 26A51,  
              05C35,  
              05D99.} 

\begin{document}
\baselineskip 16pt

\maketitle

\begin{abstract}
We study extremal properties of the function
  $$ F(x) := \min\{k\|x\|^{1-1/k}\colon k\ge 1\},\ x\in[0,1], $$
where $\|x\|=\min\{x,1-x\}$. In particular, we show that $F$ is the pointwise
largest function of the class of all real-valued functions $f$ defined on the
interval $[0,1]$, and satisfying the relaxed convexity condition
  $$ f(\lam x_1+(1-\lam)x_2) \le \lam f(x_1)+(1-\lam)f(x_2)+|x_2-x_1|,
                                                 \ x_1,x_2,\lam\in[0,1] $$
and the boundary condition $\max\{f(0),f(1)\}\le 0$.

As an application, we prove that if $A$ and $S$ are subsets of a finite
abelian group $G$, such that $S$ is generating and all of its elements have
order at most $m$, then the number of edges from $A$ to its complement $G\stm
A$ in the directed Cayley graph induced by $S$ on $G$ is
  $$ \prt_S(A) \ge \frac1m\,|G|\, F(|A|/|G|). $$
\end{abstract}

\section{Summary of results}

In this section we discuss our principal results; the proofs are
presented in Sections~\refs{proof-of-T1}--\refs{more-proofs}.

The central character of this paper is the function
  $$ F(x) := \min \{ k\|x\|^{1-1/k}\colon k\ge 1\}, \ x\in[0,1], $$
where $k$ runs over positive integers, and $\|x\|$ denotes the distance from
$x$ to the nearest integer; that is, $\|x\|=\min\{x,1-x\}$ for $x\in[0,1]$.
More explicitly, letting $\bet_0=1/2$ and $\bet_k:=(1+1/k)^{-k(k+1)}$ for
integer $k\ge 1$ (so that $1/4=\bet_1>\bet_2>\dotsb$), we have
$F(x)=kx^{1-1/k}$ whenever $\bet_k\le x\le\bet_{k-1}$, and $F(1-x)=F(x)$. The
graphs of the function $F$ and the functions $kx^{1-1/k}$ for $k\in\{1,2,3\}$
are presented in Figure~\ref{f:fig}.

\begin{figure}[h]\label{f:fig}
\centering
\includegraphics[trim=1.5mm 1.5mm 1.5mm 1.5mm,clip,
  height=0.40\textwidth,width=0.43\textwidth]{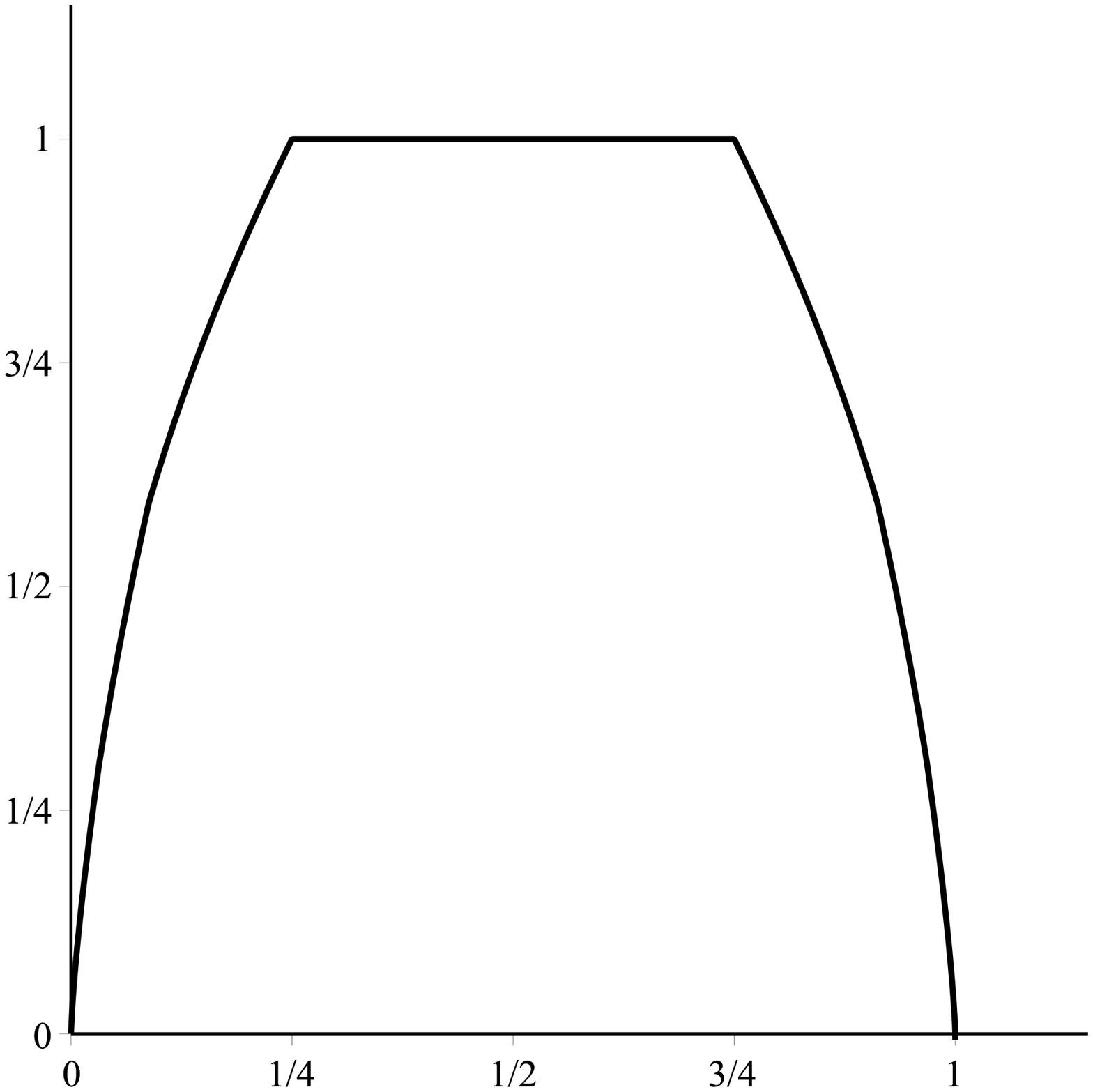}
\hskip 0.08\textwidth
\includegraphics[trim=1.5mm 1.5mm 1.5mm 1.5mm,clip,
  height=0.40\textwidth,width=0.40\textwidth]{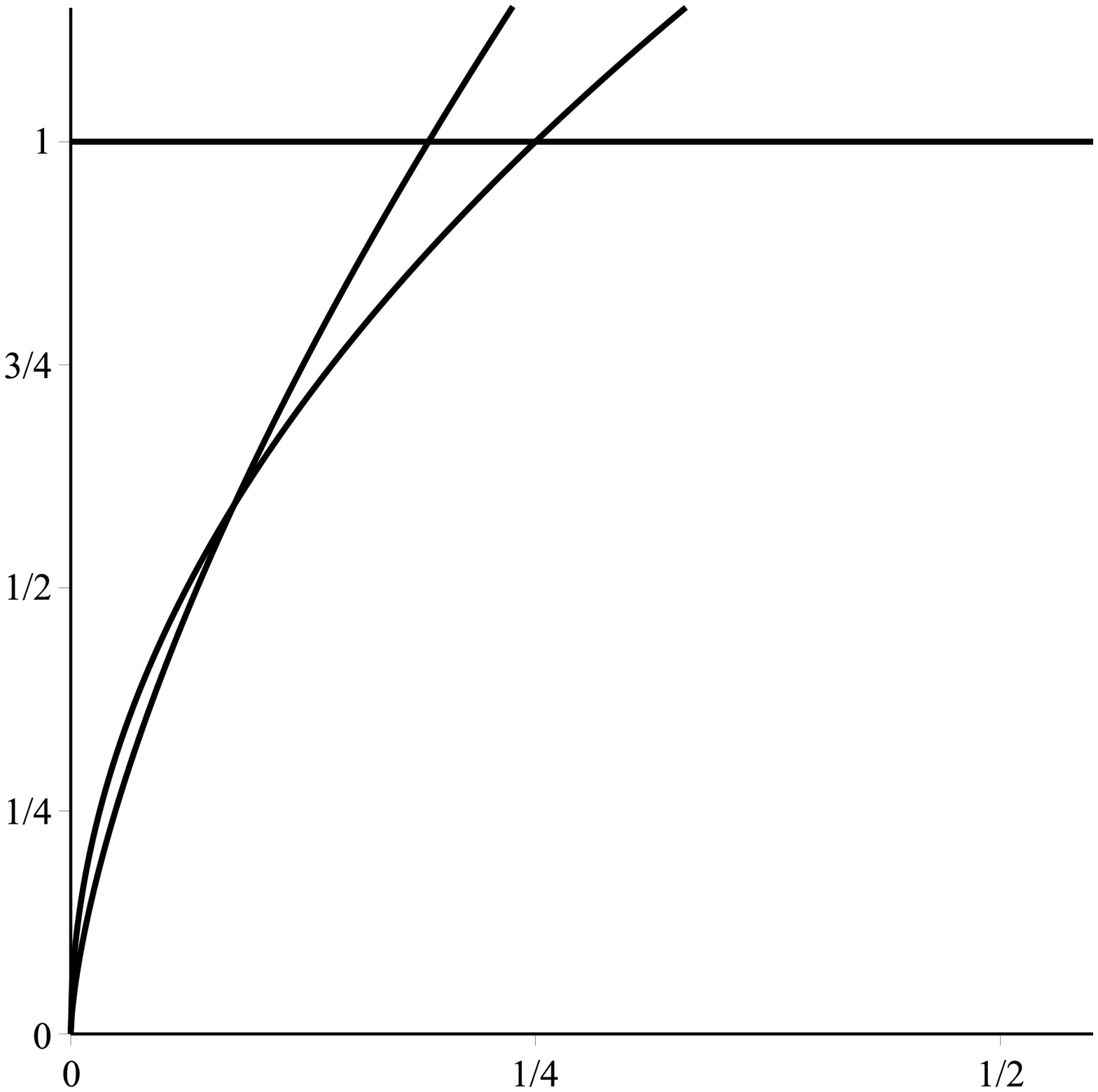}
\caption{The graphs of the functions $F$ and $kx^{1-1/k}$ for
  $k\in\{1,2,3\}$.}
\end{figure}

Recall that for $c,p>0$, a real-valued function $f$ defined on a convex
subset of a (real) normed vector space is called $(c,p)$-convex if it
satisfies
  $$ f(\lam x_1+(1-\lam)x_2) \le \lam f(x_1)+(1-\lam)f(x_2)+ c\|x_2-x_1\|^p $$
for all $x_1$ and $x_2$ in the domain of $f$, and all $\lam\in[0,1]$. We
denote by $\cF$ the class of all $(1,1)$-convex functions on the interval
$[0,1]$; that is, $\cF$ consists of all real-valued functions $f$, defined on
$[0,1]$ and satisfying
\begin{equation}\label{e:cF-def}
  f(\lam x_1+(1-\lam)x_2) \le \lam f(x_1)+(1-\lam) f(x_2)+|x_2-x_1|,
     \ x_1,x_2,\lam\in[0,1].
\end{equation}
Also, let $\cF_0$ be the subclass of all those functions $f\in\cF$ satisfying
the boundary condition
\begin{equation}\label{e:bound-cond}
  \max\{f(0),f(1)\} \le 0.
\end{equation}
Since $\cF$ is closed under translates by a linear function, any function
$f\in\cF$ can be forced into $\cF_0$ just by adding to it an appropriate
linear summand. Hence, studying these two classes is essentially equivalent.

It is immediate from the definition that the classes $\cF$ and $\cF_0$ are
``symmetric around $x=1/2$'' in the sense that a function $f$ belongs to the
class $\cF$ ($\cF_0$) if and only if so does the function $x\mapsto f(1-x)\
(x\in[0,1])$.

Our first principal result shows that the above-defined function $F$ lies in
the class $\cF_0$ and indeed, is the (pointwise) largest function of this
class.
\begin{theorem}\label{t:max-F}
We have $F\in\cF_0$ and $f\le F$ for every function $f\in\cF_0$.
\end{theorem}

We remark that substituting $x_1=1$ and $x_2=0$ into \refe{cF-def} shows that
all functions from the class $\cF_0$ are uniformly bounded from above, whence
the function $\sup\{f\colon f\in\cF_0\}$ is well defined. It is not difficult
to see that this function itself belongs to $\cF_0$ and is pointwise bounded
from above by the function $F$. However, proving that $F\in\cF_0$ is more
delicate.

For integer $m\ge 2$, let $\cF_m$ denote the class of all real-valued
functions, defined on the interval $[0,1]$ and satisfying the boundary
condition \refe{bound-cond} and the estimate
\begin{equation}\label{e:m}
  f\( \frac{x_1\longp x_m}m \) \le \frac{f(x_1)\longp f(x_m)}{m} + (x_m-x_1),
\end{equation}
for all $x_1\longc x_m\in[0,1]$ with $\min_i x_i=x_1$ and $\max_i x_i=x_m$.
These classes were introduced in \refb{l}, except that the functions from the
class $\cF_2$ (up to the normalization \refe{bound-cond}) are known as
$(1,1)$-\emph{midconvex} and under this name have been studied in a number of
papers; see, for instance, \refb{ttb}. As shown in \cite[Lemma~3]{b:l}, every
concave function from the class $\cF_0$ belongs to all classes $\cF_m$. Thus,
from Theorem~\reft{max-F} we get
\begin{corollary}\label{c:FinFm}
We have $F\in\cF_m$ for all $m\ge 2$.
\end{corollary}

As an application, we establish a result from the seemingly unrelated realm
of edge isoperimetry.

The edge-isoperimetric problem for a graph on the vertex set $V$ is to find,
for every non-negative integer $n\le |V|$, the smallest possible number of
edges between an $n$-element set of vertices and its complement in $V$. We
refer the reader to the survey of Bezrukov \refb{be} and the monograph of
Harper \refb{h} for the history, general perspective, numerous results,
variations, and further references on this and related problems.

In the present paper we are concerned with the situation where the graph
under consideration is a Cayley graph on a finite abelian group. We use the
following notation. Given two subsets $A,S\seq G$ of a finite abelian group
$G$, we consider the directed Cayley graph, induced by $S$ on $G$, and we
write $\prt_S(A)$ for the edge-boundary of $A$; that is, $\prt_S(A)$ is
number of edges in this graph from an element of $A$ to an element in its
complement $G\stm A$:
  $$ \prt_S(A) := |\{ (a,s)\in A\times S\colon a+s\notin A \}|. $$
Equivalently, $\prt_S(A)$ is the number of edges between $A$ and $G\stm A$ in
the \emph{undirected} Cayley graph induced on $G$ by the set $S\cup(-S)$.

As a consequence of Corollary~\refc{FinFm} (thus, ultimately, of
Theorem~\reft{max-F}), we prove
\begin{theorem}\label{t:isoper}
Let $A$ and $S$ be subsets of a finite abelian group $G$, of which $S$ is
generating. If $m$ is a positive integer such that the order of every element
of $S$ does not exceed $m$, then
  $$ \prt_S(A) \ge \frac 1m\,|G| F(|A|/|G|). $$
\end{theorem}

Our proof of Theorem~\reft{isoper} is a variation of the argument used in
\cite[Theorem~5]{b:l} where a slightly weaker estimate is established under
the stronger assumption that \emph{all elements of $G$} have order at most
$m$.

In the special particular case where $G$ is a homocyclic group of exponent
$m$, and $S\seq G$ is a standard generating subset, Theorem~\reft{isoper}
gives a result of Bollob\'as and Leader \cite[Theorem~8]{b:bl}.

The estimate of Theorem~\reft{isoper}, in general, fails to hold for
non-abelian groups. For instance, if $G$ is the symmetric group of order
$|G|=6$, and $S\seq G$ consists of two involutions, then the Cayley graph
induced on $G$ by $S$ is a bi-directional cycle of length $6$. Consequently,
for a non-empty proper subset $A\seq G$ inducing a path in this cycle, one
has
  $$ \prt_S(A) = 2 < \frac12\,|G|\,F(|A|/|G|) $$
(as $F(n/6)=\sqrt{2/3}$ for $n\in\{1,5\}$ and $F(n/6)=1$ for
$n\in\{2,3,4\}$).

It would be interesting to investigate the sharpness of the estimate of
Theorem~\reft{isoper} and to determine whether the function $F$ in its
right-hand side can be replaced with a larger function.

Back to the class $\cF$, from Theorem \reft{max-F} we derive the following
result showing that, somewhat surprisingly, any function satisfying
\refe{cF-def} must actually satisfy a stronger inequality.
\begin{theorem}\label{t:strong}
For any function $f\in\cF$, and any $x_1,x_2,\lam\in[0,1]$, we have
  $$ f(\lam x_1+(1-\lam)x_2) \le \lam f(x_1)+(1-\lam)f(x_2) +
                                                      F(\lam)|x_2-x_1|. $$
\end{theorem}
Substituting $x_1=1$, $x_2=0$, and $f=F$ into the inequality of
Theorem~\reft{strong}, we conclude that the factor $F(\lam)$ in the
right-hand side is optimal, and the function $F$ cannot be replaced with a
larger function.

As shown in \cite[Theorem~6]{b:l}, for each $m\ge 2$, the functions
$F_m:=\sup\{f\colon f\in\cF_m\}$ are well-defined and belong themselves to
the classes $\cF_m$, and \cite[Theorem~5]{b:l} gives a lower bound for the
edge-boundary $\prt_S(A)$ in terms of these functions. In this context it is
natural to investigate the infimum $\inf\{F_m\colon m\ge 2\}$. It is somewhat
surprising that this infimum can be found explicitly, even though the
individual functions $F_m$ are known for $2\le m\le 4$ only
(see~\cite[Theorem~7]{b:l}).
\begin{theorem}\label{t:inf-Fm}
We have
  $$ \inf\{F_m\colon m\ge 2\}=F. $$
\end{theorem}

As mentioned above, \cite[Lemma~3]{b:l} says that every concave function from
the class $\cF_0$ belongs to all the classes $\cF_m$. For the proof of
Theorem~\reft{inf-Fm} we need the converse assertion, which turns out to be
true even with the concavity assumption dropped.
\begin{lemma}\label{l:ConvL3}
If $f\in\cF_m$ for all $m\ge 2$, then $f\in\cF_0$.
\end{lemma}

Clearly, every convex function lies in the class $\cF$. The last result of
our paper shows that all ``sufficiently flat'' \emph{concave} functions also
lie in $\cF$; this complements in a sense Theorem~\reft{max-F} which implies
that every concave function from the class $\cF$ is ``flat''.

Let $\cC$ be the class of all functions, defined and concave on the interval
$[0,1]$ and vanishing at the endpoints of this interval.
\begin{theorem}\label{t:flat}
If $f\in\cC$ and $f(x)\le4x(1-x)$ for all $x\in[0,1]$, then $f\in\cF_0$.
Moreover, the function $4x(1-x)$ is best possible here in the following
sense: if $h\in\cC$ has the property that for any function $f\in\cC$ with
$f\le h$ we have $f\in\cF_0$, then $h(x)\le 4x(1-x)$ for all $x\in[0,1]$.
\end{theorem}

We now turn to the proofs of the results discussed above. We prove
Theorems~\reft{max-F} and~\reft{isoper} in Sections~\refs{proof-of-T1}
and~\refs{proof-of-T2}, respectively, and the proofs of
Theorems~\reft{strong},~\reft{inf-Fm}, and~\reft{flat} and
Lemma~\refl{ConvL3} are presented in Section~\refs{more-proofs}. Concluding
remarks are gathered in Section~\refs{remarks}.

\section{The proof of Theorem~\reft{max-F}}\label{s:proof-of-T1}

First, we show that for any function $f\in\cF_0$, and any real $x\in[0,1]$
and integer $k\ge 1$, one has $f(x)\le k\|x\|^{1-1/k}$; this will prove the
second assertion of the theorem. By symmetry, $x\le 1/2$ can be assumed
without loss of generality. Applying \refe{cF-def} with $x_1=\lam^{k-1}$ and
$x_2=0$, we get
  $$ f(\lam^k)\le \lam f(\lam^{k-1})+\lam^{k-1}; $$
that is,
  $$ \frac{f(\lam^k)}{\lam^k}
                     \le \frac{f(\lam^{k-1})}{\lam^{k-1}} + \frac1\lam. $$
Iterating, we obtain $f(\lam^k)/\lam^k\le k/\lam$ whence, substituting
$\lam=x^{1/k}$,
  $$ f(x) \le k \lam^{k-1} = kx^{1-1/k}, $$
as wanted.

We now prove that $F$ satisfies \refe{cF-def}, and hence $F\in\cF_0$. The
proof is based on the following lemma showing that if a concave and
continuous function satisfies \refe{cF-def} whenever $x_2\in\{0,1\}$, then it
actually satisfies \refe{cF-def} for all $x_1,x_2\in[0,1]$.
\begin{lemma}\label{l:redend}
Suppose that the function $f$ is concave and continuous on the interval
$[0,1]$. In order for \refe{cF-def} to hold for all $\lam,x_1,x_2\in[0,1]$,
it is sufficient that it holds for all $x_1,\lam\in[0,1]$ and
$x_2\in\{0,1\}$.
\end{lemma}

To avoid interrupting the flow of exposition, we proceed with the proof of
Theorem~\reft{max-F}, postponing the proof of Lemma~\refl{redend} to the end
of this section.

As Lemma~\refl{redend} shows, it suffices to establish \refe{cF-def} with
$f=F$ and $x_2\in\{0,1\}$. Indeed, the case $x_2=1$ reduces easily to that
where $x_2=0$ using the symmetry of $F$ around the point $1/2$. Thus, $x_2=0$
can be assumed, and in view of $F(0)=0$, renaming the remaining variable, we
have to prove that
\begin{equation}\label{e:toprove}
  F(\lam x) \le \lam F(x) + x,\  x,\lam\in[0,1].
\end{equation}
The situation where $x\in\{0,1\}$ is immediate, and we therefore assume that
$0<x<1$.

Addressing first the case $x\le 1/2$, we find $k\ge 1$ such that
$F(x)=kx^{1-1/k}$, and notice that $F(\lam x)\le(k+1)(\lam x)^{1-1/(k+1)}$ by
the definition of the function $F$; consequently, it suffices to prove that
  $$ (k+1)(\lam x)^{1-1/(k+1)} \le k\lam x^{1-1/k} + x. $$
Dividing through by $x$, substituting $t:=\lam^{1/(k+1)}x^{-1/(k(k+1))}$, and
rearranging the terms gives this inequality the shape
  $$ kt^{k+1} - (k+1)t^k + 1 \ge 0, $$
and to complete the proof it remains to notice that the left-hand side
factors as
  $$ (t-1)^2(kt^{k-1}+(k-1)t^{k-2}\longp 2t+1). $$

The case $1/2\le x<1$ reduces to that where $0<x\le 1/2$ as follows. Let
$x':=1-x$, so that $0<x'\le 1/2$. Assuming that \refe{toprove} fails to hold,
we get
  $$ F(\lam x) > x $$
and also
\begin{equation}\label{e:loc42}
  \lam F(x) < F(\lam x) - x \le x',
\end{equation}
which, in view of $F(x')=F(x)$, jointly yield
  $$ \frac{F(x')}{x'} < \frac1\lam < \frac{F(\lam x)}{\lam x}. $$
Since $F(z)/z$ is a decreasing function of $z$ (which is obvious for
$z\in[1/2,1]$, and follows directly from the definition of the function $F$
for $z\in[0,1/2]$), we conclude that $x'>\lam x$. Thus, $\lam x/x'<1$, and
recalling that $x'\le 1/2$, by what we have shown above,
  $$ F(\lam x) = F\left(\frac{\lam x}{x'}\cdot x'\right)
                                    \le \frac{\lam x}{x'}\,F(x') + x'. $$
Hence, from the assumption that \refe{toprove} is false,
  $$ \lam F(x) + x < \frac{\lam x}{x'}\,F(x') + x', $$
which can be written as
  $$ \frac{x-x'}{x'}\,\lam F(x) > x-x', $$
contradicting \refe{loc42}.

\medskip
It remains to prove Lemma~\refl{redend}. The proof uses the well-known fact
that a strictly concave function is unimodal; the specific version we need
here is that if $f$ is continuous and strictly concave on the closed interval
$[u,v]$, then either it is strictly monotonic on $[u,v]$, or there exists
$w\in(u,v)$ such that $f$ is strictly increasing on $[u,w]$ and strictly
decreasing on $[w,v]$. As a result, the minimum of a function, strictly
concave on a closed interval, is attained at one of the endpoints of the
interval. We also need
\begin{claim}\label{m:conv-inv}
Suppose that the function $f$ is defined and strictly concave on a closed
interval $[u,v]$. If $f$ is monotonically increasing on $[u,v]$, then its
inverse is strictly convex. If $f$ is monotonically decreasing on $[u,v]$,
then its inverse is strictly concave.
\end{claim}

\begin{proof}[Proof of Lemma~\refl{redend}]
Suppose first that $f$ is \emph{strictly} concave on $[0,1]$, and assume that
\refe{cF-def} holds for all $\lam,x_1\in[0,1]$ and $x_2\in\{0,1\}$; hance, by
symmetry, also for all $\lam,x_2\in[0,1]$ and $x_1\in\{0,1\}$.

Given $x_1,x_2\in[0,1]$ with $x_1<x_2$, let $k:=(f(x_2)-f(x_1))/(x_2-x_1)$,
and consider the auxiliary function $g(x):=f(x)-kx$ (depending on $x_1$ and
$x_2$). Furthermore, set $\bet:=g(x_i)\ (i\in\{1,2\})$ and
$B:=\max\{g(x)\colon x\in[0,1]\}$. Notice, that $g(x_1)=g(x_2)$, in
conjunction with the strict concavity of $g$, implies $B>\max\{g(0),g(1)\}$.
For any $\lam\in[0,1]$ we have then
\begin{multline*}
  f(\lam x_1+(1-\lam)x_2) - \lam f(x_1) - (1-\lam)f(x_2) \\
      = g(\lam x_1+(1-\lam)x_2) - \lam g(x_1) - (1-\lam )g(x_2) \le B - \bet,
\end{multline*}
and therefore it suffices to show that
\begin{equation}\label{e:x2x1betB}
  x_2 - x_1 + \bet \ge B
\end{equation}
for any $0<x_1<x_2<1$ and with $\bet=\bet(x_1,x_2)$ and $B=B(x_1,x_2)$
defined as above.

We now put the reasoning onto its head. Suppose that a real $k$ is fixed so
that, if $g(x)=f(x)-kx$ and $B=\max\{g(x)\colon x\in[0,1]\}$, then
$\max\{g(0),g(1)\}<B$. Let $w\in(0,1)$ be defined by $g(w)=B$. By the
intermediate value property and monotonicity of $g$ on each of the intervals
$[0,w]$ and $[w,1]$, to any given $\bet$ with $\max\{g(0),g(1)\}\le\bet\le B$
there corresponds then a unique pair $(x_1,x_2)$ with $g(x_1)=g(x_2)=\bet$
and $0\le x_1\le w\le x_2\le 1$. As the above argument shows, to complete the
proof (under the \emph{strict} concavity assumption) it suffices to establish
\refe{x2x1betB} with $x_1$ and $x_2$ understood as functions of the variable
$\bet$ ranging from $\max\{g(0),g(1)\}$ to $B$. Since $x_1$ and $x_2$ are
actually inverses of the function $g$ restricted to the appropriate
intervals, by Claim~\refm{conv-inv}, $x_1$ is convex and continuous, and
$x_2$ is concave and continuous, so that $x_2-x_1+\bet$ is concave and
continuous and consequently, \refe{x2x1betB} will follow once we obtain it
for $\bet=\max\{g(0),g(1)\}$ and also for $\bet=B$. The latter case (with
equality sign) is immediate from $x_2\ge x_1$. For the former case, we notice
that if $\bet=\max\{g(0),g(1)\}$, then $x_1(1-x_2)=0$, whence, having
$t\in[0,1]$ defined by $w=tx_1+(1-t)x_2$, we get
\begin{align*}
  x_2 - x_1 + \bet
    &= x_2 - x_1 + B
             - \big(g(tx_1+(1-t)x_2) - tg(x_1) - (1-t)g(x_2)\big) \\
    &= x_2 - x_1 + B
             - \big(f(tx_1+(1-t)x_2) - tf(x_1) - (1-t)f(x_2)\big) \\
    &\ge B
\end{align*}
by the assumption of the lemma (and the remark at the very beginning of the
proof).

Finally, suppose that $f$ is concave but, perhaps, not \emph{strictly}
concave on $[0,1]$. For $\eps\in(0,1)$ let $f_\eps(x):=(f(x)-\eps
x^2)/(1+\eps)$ and define
\begin{align*}
  \Del(x_1,x_2,\lam)
           &:= f(\lam x_1+(1-\lam)x_2)
                             - \lam f(x_1) - (1-\lam) f(x_2) - |x_2-x_1| \\
\intertext{and}
  \Del_\eps(x_1,x_2,\lam)
           &:= f_\eps(\lam x_1+(1-\lam)x_2)
                     - \lam f_\eps(x_1) - (1-\lam) f_\eps(x_2) - |x_2-x_1|.
\end{align*}
A straightforward computations confirms that
\begin{equation}\label{e:Dellam}
  \Del(x_1,x_2,\lam) = (1+\eps)\Del_\eps(x_1,x_2,\lam) + \eps
                           |x_2-x_1| \big( 1-\lam(1-\lam)|x_2-x_1| \big).
\end{equation}
Consequently, if $\Del(x_1,x_2,\lam)\le 0$ when $x_1,\lam\in[0,1]$ and
$x_2\in\{0,1\}$, then also $\Del_\eps(x_1,x_2,\lam)\le 0$ under the same
assumptions. Since $f_\eps$ is \emph{strictly} concave (as it is easy to
verify), we conclude that $\Del_\eps(x_1,x_2,\lam)\le 0$ actually holds for
all $x_1,x_2,\lam\in[0,1]$. Now \refe{Dellam} shows that
$\Del(x_1,x_2,\lam)\le\eps$ for all $x_1,x_2,\lam\in[0,1]$, and since $\eps$
can be chosen arbitrarily small, we have indeed $\Del(x_1,x_2,\lam)\le 0$.
\end{proof}

\section{The proof of Theorem~\reft{isoper}}\label{s:proof-of-T2}

We use induction on $|G|$. Without loss of generality, we assume that $S$ is
a minimal (under inclusion) generating subset. Fix an element $s_0\in S$ and
write $S_0:=S\stm\{s_0\}$. If $S_0=\est$, then $G$ is cyclic with $|G|$ being
equal to the order of $s_0$, whence $|G|\le m$ and the assertion follows from
$F\le 1$. Assuming now that $S_0\ne\est$, let $H$ be the subgroup of $G$,
generated by $S_0$; thus, $H$ is proper and non-trivial. Since the quotient
group $G/H$ is cyclic, generated by $s_0+H$, its order $l:=|G/H|$ does not
exceed $m$. For $i=1\longc l$ set $A_i:=A\cap(is_0+H)$ and $x_i:=|A_i|/|H|$.

Fix $i\in[1,l]$. By the induction hypothesis (as applied to the subset
$(A-is_0)\cap H$ of the group $H$ with the generating subset $S_0$), the
number of edges from an element of $A_i$ to an element of $(is_0+H)\stm A$ is
at least $\frac1m\,|H|F(x_i)$. Furthermore, the number of edges from $A_i$ to
$((i+1)s_0+H)\stm A$ is at least
  $$ \max \{ |A_i|-|A_{i+1}| , 0 \} \\
    = |H| \max\{ x_i-x_{i+1}, 0 \}
                = \frac12\,|H|\big( |x_i-x_{i+1}| + x_i-x_{i+1} \big) $$
(where $x_{i+1}$ is to be replaced with $x_1$ for $i=l$). It follows that
\begin{multline*}
  \prt_S(A) \ge \frac1m\,|H|\,\big(F(x_1)\longp F(x_l)\big) \\
            + \frac12\, |H|\big(|x_1-x_2|\longp |x_{l-1}-x_l|+|x_l-x_1|\big).
\end{multline*}
Choose $i,j\in[1,l]$ so that $x_i$ is the smallest, and $x_j$ the largest of
the numbers $x_1\longc x_l$. From the triangle inequality,
  $$ |x_1-x_2|\longp |x_{l-1}-x_l|+|x_l-x_1| \ge 2(x_j-x_i), $$
whence
\begin{align*}
  \prt_S(A)
      &\ge \frac1m\,|G|\,\frac{F(x_1)\longp F(x_l)}l  + |H|(x_j-x_i) \\
      &\ge \frac1m\,|G|\,\( \frac{F(x_1)\longp F(x_l)}l + (x_j-x_i) \). \\
\intertext{Recalling that $F\in\cF_l$ by Corollary~\refc{FinFm}, we conclude
           that}
  \prt_S(A)
      &\ge \frac1m\,|G|\,F\( \frac{x_1\longp x_l}l \)  \\
      &= \frac1m\,|G|\,F(|A|/|G|),
\end{align*}
as wanted.

\section{The proofs of Lemma~\refl{ConvL3} and
  Theorems~\reft{strong},~\reft{inf-Fm}, and~\reft{flat}}\label{s:more-proofs}

\begin{proof}[Proof of Theorem~\reft{strong}]
Fix $x_1,x_2\in[0,1]$ with $x_1<x_2$ and consider the function $\phi$ defined
by
  $$ \phi(\lam ) := f(\lam x_1+(1-\lam )x_2)-\lam f(x_1)-(1-\lam )f(x_2),
                                                          \ \lam \in[0,1]. $$
Clearly, we have $\phi(0)=\phi(1)=0$, and for any $u,v,t\in[0,1]$ with $u<v$,
using the fact that $f\in\cF$, we get
\begin{align*}
  \phi\big(tu +&(1-t)v\big) \\
    &= f\big((tu+(1-t)v)\,(x_1-x_2)+x_2\big)
                            - (tu+(1-t)v)\,(f(x_1)-f(x_2)) - f(x_2) \\
    &= f\big( t(u(x_1-x_2)+x_2) + (1-t)(v(x_1-x_2)+x_2) \big) \\
    &\phantom{f\big( t(u(x_1-x_2)+\;}
             - t(uf(x_1)+(1-u)f(x_2)) - (1-t)(vf(x_1)+(1-v)f(x_2)) \\
    &\le tf(u(x_1-x_2)+x_2) + (1-t)f(v(x_1-x_2)+x_2) + (v-u)(x_2-x_1) \\
    &\phantom{f\big( t(u(x_1-x_2)+\;}
             - t(uf(x_1)+(1-u)f(x_2)) - (1-t)(vf(x_1)+(1-v)f(x_2)) \\
    &= t\phi(u)+(1-t)\phi(v)  + (v-u)(x_2-x_1).
\end{align*}
Hence, $(x_2-x_1)^{-1}\phi\in\cF_0$, and therefore $\phi\le(x_2-x_1)F$ by
Theorem~\reft{max-F}; that is,
  $$ f(\lam x_1+(1-\lam)x_2)-\lam f(x_1)-(1-\lam)f(x_2)
                                                  \le F(\lam) (x_2-x_1) $$
for all $x_1,x_2,\lam\in[0,1]$ with $x_1<x_2$.
\end{proof}

\begin{proof}[Proof of Lemma~\refl{ConvL3}]
Aiming at a contradiction, suppose that $f\in\cap_{m\ge 2}\cF_m$, but
$f\notin\cF_0$. By the latter assumption, there exist $\lam,x_1,x_2\in[0,1]$
with $x_1\le x_2$ such that
\begin{equation}\label{e:aim-contr}
  f(\lam x_1+(1-\lam)x_2) > \lam f(x_1)+(1-\lam)f(x_2) + (x_2-x_1).
\end{equation}
Clearly, we have $\lam\in(0,1)$. This implies
 $x:=\lam x_1+(1-\lam)x_2\in(0,1)$, whence $f$ is continuous at $x$ by
\cite[Lemma~4]{b:l} (which says that all functions from the classes $\cF_m$
are continuous on $(0,1)$). It follows that there is a \emph{rational}
$\lam\in(0,1)$ for which \refe{aim-contr} holds true; say, $\lam=u/m$ with
integer $0<u<m$. Now \refe{aim-contr} can be written as
  $$ f\(\frac{ux_1+(m-u)x_2}m\) > \frac{uf(x_1)+(m-u)f(x_2)}m + x_2-x_1, $$
contradicting the assumption $f\in\cF_m$.
\end{proof}

\begin{proof}[Proof of Theorem~\reft{inf-Fm}]
Let $F_0:=\inf\{F_m\colon m\ge 2\}$; our goal, therefore, is to show that
$F_0=F$. To begin with, we prove that
\begin{equation}\label{e:F0incFm}
  F_0\in\cF_m,\ m\ge 2.
\end{equation}
To this end, we fix $\eps>0$ and $x_1\longc x_m\in[0,1]$ with
 $\min_i x_i=x_1$ and $\max_i x_i=x_m$, and show that
\begin{equation}\label{e:tmp3}
  F_0\(\frac{x_1\longp x_m}m\) \le \frac{F_0(x_1)\longp F_0(x_m)}m
                                                        + (x_m-x_1) + \eps.
\end{equation}
As shown in \refb{l}, if $k$ and $l$ are integers with $k\mid l$, then
$\cF_l\seq\cF_k$, and hence $F_l\le F_k$. It follows that
  $$ F_0(x) = \lim_{l\to\infty} F_{l!}(x), \quad x\in[0,1]; $$
thus, there is an integer $l\ge m$ such that $F_{l!}(x_i)\le F_0(x_i)+\eps$
for each $i\in[1,m]$. Since $F_{l!}\in\cF_{l!}\seq\cF_m$ in view of
 $m\mid l!$, we then get
\begin{align*}
   F_0\(\frac{x_1\longp x_m}m\)
    &\le F_{l!}\(\frac{x_1\longp x_m}m\) \\
    &\le \frac{F_{l!}(x_1)\longp F_{l!}(x_m)}m + (x_m-x_1) \\
    &\le \frac{F_0(x_1)\longp F_0(x_m)}m + (x_m-x_1) + \eps,
\end{align*}
establishing \refe{tmp3}, and therefore \refe{F0incFm}.

To complete the proof we notice that \refe{F0incFm} and Lemma~\refl{ConvL3}
yield $F_0\in\cF_0$, whence, by Theorem~\reft{max-F},
\begin{equation}\label{e:F0<F}
  F_0\le F.
\end{equation}
On the other hand, since $F\in\cF$ is concave, by \cite[Lemma~3]{b:l} we have
$F\in\cF_m$ for every $m\ge 2$. It follows that $F\le F_m$ for every $m\ge
2$, implying
\begin{equation}\label{e:F<F0}
  F\le F_0.
\end{equation}
Comparing \refe{F0<F} and \refe{F<F0}, we get the assertion.
\end{proof}

\begin{proof}[Proof of Theorem~\reft{flat}]
For the first assertion of the theorem we show that \refe{cF-def} holds true,
provided that $f\in\cC$ and $f(x)\le4x(1-x),\ x\in[0,1]$. Using symmetry and
disposing of the trivial cases, we assume that $x_1<x_2$ and $0<\lam<1$.
Furthermore, letting $x_0:=\lam x_1+(1-\lam)x_2$, we rewrite the inequality
to prove as
\begin{equation}\label{e:to-prove}
  f(x_0) \le \frac{x_2-x_0}{x_2-x_1}\,f(x_1)
                                  + \frac{x_0-x_1}{x_2-x_1}\,f(x_2) + x_2-x_1.
\end{equation}

By concavity, we have
  $$ f(x_0) \le \frac{x_0}{x_1}\, f(x_1) $$
(as the point $(x_1,f(x_1))$ lies above the segment joining the points
$(0,0)$ and $(x_0,f(x_0))$, and
  $$ f(x_0) \le \frac{1-x_0}{1-x_2}\,f(x_2) $$
(as $(x_2,f(x_2))$ is above the segment joining $(x_0,f(x_0))$ and $(1,0)$).
Also, by the assumptions,
  $$ f(x_0) \le 4x_0(1-x_0). $$
Comparing with \refe{to-prove}, we see that it suffices to prove that
\begin{multline*}
  \min \left\{ \frac{x_0}{x_1}\, f(x_1),
          \ \frac{1-x_0}{1-x_2}\,f(x_2),\ 4x_0(1-x_0) \right\} \\
                   \le \frac{x_2-x_0}{x_2-x_1}\,f(x_1)
                             + \frac{x_0-x_1}{x_2-x_1}\,f(x_2) + x_2-x_1.
\end{multline*}
Assuming for a contradiction that this is wrong, after tedious, but routine
algebraic manipulations we then derive
\begin{align*}
  x_2f(x_1) - x_1f(x_2) &> x_1 \frac{(x_2-x_1)^2}{x_0-x_1}, \\
  (x_2-1)f(x_1) + (1-x_1)f(x_2) &> (1-x_2)\frac{(x_2-x_1)^2}{x_2-x_0}, \\
\intertext{and}
  (x_0-x_2)f(x_1) + (x_1-x_0)f(x_2) &> (x_2-x_1)^2 - 4x_0(1-x_0)(x_2-x_1).
\end{align*}
Multiplying the first of these inequalities by $1-x_0$ and the second by
$x_0$, and adding up the resulting estimates and the third inequality, we get
  $$ 4x_0(1-x_0)(x_2-x_1)
          > \left( \frac{x_1(1-x_0)}{x_0-x_1}
                   + \frac{x_0(1-x_2)}{x_2-x_0} + 1 \right) (x_2-x_1)^2. $$
It is easily verified that this simplifies to
  $$ \frac{(x_2-x_1)^2}{(x_0-x_1)(x_2-x_0)} < 4 $$
and further to
  $$ 4x_0^2 - 4x_0(x_1+x_2) + (x_1+x_2)^2 < 0, $$
which cannot hold since the left-hand side is a square.

To prove the second assertion, suppose that $h$ is a concave function on
$[0,1]$ with $h(0)=h(1)=0$ and $h(x_0)>4x_0(1-x_0)$ for some $x_0\in(0,1)$,
and let
  $$ f(x) := \begin{cases}
               \frac{h(x_0)}{x_0}\,x\ &\text{if}\ 0\le x\le x_0, \\
               \frac{h(x_0)}{1-x_0}\,(1-x)\ &\text{if}\ x_0\le x\le 1;
             \end{cases} $$
thus, $f$ is a concave function on $[0,1]$ with $f(0)=f(1)=0$, and $f\le h$
by concavity of $h$. We show that, on the other hand, $f\notin\cF_0$, and
indeed, assuming for definiteness $x_0\le 1/2$, that
  $$ f(x_0) > \frac12\,f(0) + \frac12\,f(2x_0) + 2x_0. $$
To this end we just plug in the definition of $f$ and rewrite this inequality
as
  $$ 2x_0 < h(x_0) - \frac12\,\frac{h(x_0)}{1-x_0}\,(1-2x_0)
                                                 = \frac{h(x_0)}{2(1-x_0)}, $$
which is equivalent to the assumption $h(x_0)>4x_0(1-x_0)$.
\end{proof}

\section{Concluding remarks}\label{s:remarks}


It seems natural to extend the class $\cF$ by considering, for every finite
closed interval $I$ and every constant $c>0$, the class $\cF(I,c)$ of those
real-valued functions $f$ defined on $I$ and satisfying
  $$ f(\lam x_1+(1-\lam)x_2)\le \lam f(x_1)+(1-\lam)f(x_2) + c|x_2-x_1| $$
for all $x_1,x_2\in I$ and $\lam\in[0,1]$. This, however, does not lead to
any principally new results, as one has $f\in\cF(I,c)$ if and only if
$(c|I|)^{-1}f\circ\phi\in\cF$, where $\phi$ is a linear bijection of $[0,1]$
onto $I$; that is, $\cF(I,c)$ is obtained from $\cF$ by a simple linear
scaling. In particular, $c|I|F\circ\phi^{-1}$ is the largest function of the
subclass of all functions from $\cF(I,c)$, non-positive at the endpoints of
$I$. Also, as a corollary of Theorem~\reft{strong}, if $f\in\cF(I,c)$, then
  $$ f(\lam x_1+(1-\lam)x_2) \le \lam f(x_1)+(1-\lam)f(x_2) +
                                                      cF(\lam)|x_2-x_1| $$
for all $x_1,x_2\in I$ and $\lam\in[0,1]$.

In contrast, it might be interesting to extend the results of this paper, and
in particular Theorem~\reft{max-F}, onto the class of all $(c,p)$-convex
functions on a given closed interval, for every fixed $p>0$. Normalization
reduces this to studying the class $\cF_0^{(p)}$ of all real-valued functions
on $[0,1]$, satisfying the boundary condition \refe{bound-cond} and the
inequality
  $$ f(\lam x_1+(1-\lam)x_2) \le \lam f(x_1)+(1-\lam)f(x_2)+|x_2-x_1|^p, $$
for all $x_1,x_2,\lam\in[0,1]$. It is not difficult to see that the function
$F^{(p)}:=\sup\{f\colon f\in\cF^{(p)}\}$ is well defined and lies itself in
the class $\cF^{(p)}$, and that for any $f\in\cF^{(p)}$ we have
  $$ f(\lam x_1+(1-\lam)x_2)
               \le \lam f(x_1)+(1-\lam)f(x_2) + F^{(p)}(\lam)|x_2-x_1|^p, $$
for all $x_1,x_2,\lam\in[0,1]$; moreover, $F^{(p)}$ is the largest function
with this property. In view of Theorem~\reft{max-F}, one can expect that,
perhaps, an explicit expression for the functions $F^{(p)}$ can be found.

\section*{Acknowledgement}
The author is grateful to Eugen Ionescu for his interest and useful
discussions.

\bigskip

\end{document}